\documentclass[pagesize,pdftex]{scrartcl}
\usepackage{latexsym} 
\usepackage[intlimits]{amsmath}
\usepackage{amsthm}
\usepackage{amsfonts}
\usepackage{amssymb}
\usepackage{amsxtra}
\usepackage{amscd}
\usepackage{ifthen}
\usepackage{graphicx}
\usepackage[shortlabels]{enumitem}
\usepackage{mathrsfs}
\usepackage[pagebackref=true]{hyperref}

\hypersetup{
pdfauthor={Mads Kyed},
pdftitle={A fundamental solution to the time-periodic Stokes equations},
breaklinks=true,
colorlinks=true,
linkcolor=blue,
citecolor=blue,
urlcolor=blue,
filecolor=blue,
}

\numberwithin{equation}{section} 
\setkomafont{title}{\normalfont}

\newenvironment{pdeq}{ \left\{ \begin{aligned}}{\end{aligned}\right.}

\newcommand{\np}[1]{(#1)}
\newcommand{\nb}[1]{[#1]}
\newcommand{\bp}[1]{\big(#1\big)}
\newcommand{\bb}[1]{\big[#1\big]}
\newcommand{\Bp}[1]{\bigg(#1\bigg)}
\newcommand{\BBp}[1]{\Bigg(#1\Bigg)}
\newcommand{\Bb}[1]{\bigg[#1\bigg]}
\newcommand{\calt}{{\mathcal T}}

\newcommand{\R}{\mathbb{R}}
\newcommand{\C}{\mathbb{C}}
\newcommand{\Z}{\mathbb{Z}}

\DeclareMathOperator{\e}{e}
\newcommand{\B}{B}

\DeclareMathOperator{\Div}{div}

\DeclareMathOperator{\impart}{Im}

\newcommand{\ra}{\rightarrow}

\newcommand{\set}[1]{\ensuremath{\{#1\}}}

\newcommand{\setc}[2]{\ensuremath{\{#1\ \lvert\ #2\}}}

\newcommand{\closure}[2]{\overline{#1}^{#2}}
\newcommand{\quotientmap}{\pi}
\newcommand{\grp}{G}
\newcommand{\dualgrp}{\widehat{G}}

\newcommand{\torus}{{\mathbb T}}

\newcommand{\idmatrix}{I}
\newcommand{\Rn}{{\R^n}}

\newcommand{\ddz}{\frac{{\mathrm d}}{{\mathrm d}z}}
\newcommand{\grad}{\nabla}

\newcommand{\dx}{{\mathrm d}x}

\newcommand{\dr}{{\mathrm d}r}

\newcommand{\ds}{{\mathrm d}s}
\newcommand{\dt}{{\mathrm d}t}

\newcommand{\dy}{{\mathrm d}y}

\newcommand{\dxi}{{\mathrm d}\xi}

\newcommand{\SR}{\mathscr{S}}

\newcommand{\TDR}{\mathscr{S^\prime}}

\newcommand{\FT}{\mathscr{F}}
\newcommand{\iFT}{\mathscr{F}^{-1}}
\newcommand{\riesztrans}{\mathfrak{R}}
\newcommand{\projsymbol}{\delta_\Z}

\newcommand{\Mmultiplier}{M}

\newcommand{\hankelone}[1]{H^{(1)}_{#1}}
\newcommand{\norm}[1]{\lVert#1\rVert}

\newcommand{\snorm}[1]{{\lvert #1 \rvert}}
\newcommand{\snorml}[1]{{\bigl\lvert #1 \big\rvert}}
\newcommand{\snormL}[1]{{\Bigl\lvert #1 \Big\rvert}}

\newcommand{\WSR}[2]{W^{#1,#2}}

\newcommand{\CR}[1]{C^{#1}}  
\newcommand{\LR}[1]{L^{#1}}

\newcommand{\LRloc}[1]{L^{#1}_{loc}} 
\newcommand{\CRi}{\CR \infty}
\newcommand{\CRci}{\CR \infty_0}

\newcommand{\LRper}[1]{L^{#1}_{\mathrm{per}}}
\newcommand{\WSRper}[2]{W^{#1,#2}_{\mathrm{per}}}

\newcommand{\CRciper}{\CR{\infty}_{0,\mathrm{per}}}

\newcommand{\vvel}{v}
\newcommand{\vpres}{p}

\newcommand{\wvel}{w}

\newcommand{\uvel}{u}

\newcommand{\upres}{\mathfrak{p}}

\newcommand{\fundsolstokes}{\varGamma_{\text{\tiny{Stokes}}}}
\newcommand{\fundsolstokesvel}{\varGamma^{\text{\tiny{S}}}}
\newcommand{\fundsolstokespres}{\gamma^{\text{\tiny{S}}}}

\newcommand{\fundsoltpstokes}{\varGamma_{\text{\tiny{TPStokes}}}}
\newcommand{\fundsoltpstokesvel}{\varGamma^{\text{\tiny{TPS}}}}
\newcommand{\fundsoltpstokespres}{\gamma^{\text{\tiny{TPS}}}}
\newcommand{\fundsolcompl}{\varGamma^\bot}

\newcommand{\fundsolssrk}{\varGamma^{k}_{\text{\tiny{SSR}}}}
\newcommand{\fundsollaplace}{\varGamma_{\text{\tiny{L}}}}
\newcommand{\tin}{\text{in }}

\newcommand{\half}{\frac{1}{2}}

\renewcommand{\epsilon}{\varepsilon}
\renewcommand{\phi}{\varphi}

\newcommand{\tay}{\calt}
\newcommand{\per}{\tay}

\newcommand{\perf}{\frac{2\pi}{\tay}}
\newcommand{\iperf}{\frac{\tay}{2\pi}}

\newcommand{\cutoff}{\chi}

\newcommand{\onedist}{1}
\newcommand{\newCCtr}[2][d]{
\newcounter{#2}\setcounter{#2}{0}
\expandafter\xdef\csname kyedtheconst#2\endcsname{#1}
}
\newcommand{\Cc}[2][nolabel]{
\stepcounter{#2}
\expandafter\ensuremath{\csname kyedtheconst#2\endcsname_{\arabic{#2}}}
\ifthenelse{\equal{#1}{nolabel}}
{}
{\expandafter\xdef\csname kyedconst#1\endcsname
{\expandafter\ensuremath{\csname kyedtheconst#2\endcsname_{\arabic{#2}}}}}
}
\newcommand{\Ccn}[2][nolabel]{
\expandafter\ensuremath{\csname kyedtheconst#2\endcsname}
\ifthenelse{\equal{#1}{nolabel}}
{}
{\expandafter\xdef\csname kyedconst#1\endcsname
{\expandafter\ensuremath{\csname kyedtheconst#2\endcsname}}}
}
\newcommand{\CcSetCtr}[2]{
\setcounter{#1}{#2}
}
\newcommand{\Cclast}[1]{
\expandafter\ensuremath{\csname kyedtheconst#1\endcsname_{\arabic{#1}}}
}
\newcommand{\Ccllast}[1]{
\addtocounter{#1}{-1}
\expandafter\ensuremath{\csname kyedtheconst#1\endcsname_{\arabic{#1}}}
\addtocounter{#1}{1}
}
\newcommand{\const}[1]{
\expandafter{\ifcsname kyedconst#1\endcsname
  \csname kyedconst#1\endcsname
\else
  \errmessage{Undefined Kyedconstant #1.}%
\fi}
}

\theoremstyle{plain}
\newtheorem{thm}{Theorem}[section]

\newtheorem{lem}[thm]{Lemma}

\theoremstyle{remark}
\newtheorem{rem}[thm]{Remark}

\begin{document}
\title{A fundamental solution to the time-periodic Stokes equations}

\author{
Mads Kyed\\ 
Fachbereich Mathematik\\
Technische Universit\"at Darmstadt\\
Schlossgartenstr. 7, 64289 Darmstadt, Germany\\
Email: \texttt{kyed@mathematik.tu-darmstadt.de}
}

\date{\today}
\maketitle

\begin{abstract}
The concept of a fundamental solution to the time-periodic Stokes equations in dimension $n\geq 2$ is introduced. A fundamental solution 
is then identified and analyzed. Integrability and pointwise estimates are established.
\end{abstract}

\noindent\textbf{MSC2010:} Primary 35Q30, 35B10, 35A08, 35E05, 76D07.\\
\noindent\textbf{Keywords:} Stokes equations, time-periodic, fundamental solution.

\newCCtr[C]{C}
\newCCtr[M]{M}
\newCCtr[B]{B}
\newCCtr[\epsilon]{eps}
\CcSetCtr{eps}{-1}

\section{Introduction}

Classically, fundamental solutions are defined for systems of linear partial differential equations in $\R^n$. Specifically, a fundamental solution to the Stokes
system ($n\geq 2$)
\begin{align}\label{intro_StokesRn}
\begin{pdeq}
&-\Delta\vvel + \grad\vpres = f && \tin\Rn, \\
&\Div\vvel =0 && \tin\Rn,
\end{pdeq}
\end{align} 
with unknowns $\vvel:\Rn\ra\Rn$, $\vpres:\Rn\ra\R$ and data $f:\Rn\ra\Rn$, is a tensor-field
\begin{align*}
\fundsolstokes:=
\begin{pmatrix}
\fundsolstokesvel_{11} & \ldots  & \fundsolstokesvel_{1n} \\
\vdots & \ddots & \vdots\\
\fundsolstokesvel_{n1} & \ldots  & \fundsolstokesvel_{nn} \\
\fundsolstokespres_{1} & \ldots  & \fundsolstokespres_{n} 
\end{pmatrix}
\in \TDR(\R^n)^{(n+1)\times n}
\end{align*}
that satisfies\footnote{We make use of the Einstein summation convention and
implicitly sum over all repeated indices.}
\begin{align}\label{intro_StokesFundEq}
\begin{pdeq}
&-\Delta\fundsolstokesvel_{ij} + \partial_i \fundsolstokespres_j = \delta_{ij}\delta_\Rn,  \\
&\partial_i\fundsolstokesvel_{ij} =0, 
\end{pdeq}
\end{align} 
where $\delta_{ij}$ and $\delta_\Rn$ denotes the Kronecker delta and delta distribution, respectively.
For arbitrary $f\in\SR(\R^n)^n$, a solution $(\vvel,\vpres)$ to \eqref{intro_StokesRn} is then given by the componentwise convolution 
\begin{align}\label{intro_ConvolutionWithStokesFundsol}
\begin{pmatrix}
\vvel \\
\vpres
\end{pmatrix} := \fundsolstokes * f,
\end{align}
which at the outset is well-defined in the sense of distributions. In the specific case of the Stokes fundamental solution $\fundsolstokes$ above, 
$\LR{q}$-integrability and pointwise decay estimates for $(\vvel,\vpres)$ can be established from \eqref{intro_ConvolutionWithStokesFundsol}. 
We refer to the standard literature such as \cite{GaldiBookNew} and \cite{VarnhornBook} for these well-known results.

The aim of this paper is to identify a fundamental solution to the \emph{time-periodic} Stokes system
\begin{align}\label{intro_TPStokesRn}
\begin{pdeq}
&\partial_t\uvel-\Delta\uvel + \grad\upres = f && \tin\Rn\times\R, \\
&\Div\uvel =0 && \tin\Rn\times\R,\\
&\uvel(x,t) = \uvel(x,t+\per)
\end{pdeq}
\end{align} 
with unknowns $\uvel:\Rn\times\R\ra\Rn$ and $\upres:\Rn\times\R\ra\R$ corresponding to time-periodic data $f:\Rn\times\R\ra\Rn$ with the same period, that is,
$f(x,t) = f(x,t+\per)$. Here $\per\in\R$ denotes the (fixed) time-period. Moreover, $x\in\R^n$ and $t\in\R$ denotes the spatial and time variable, respectively.
The main objective is to establish a framework which enables us to define and identify a fundamental solution $\fundsoltpstokes$ to \eqref{intro_TPStokesRn}
with the property that a solution $(\uvel,\upres)$ is given by a convolution
\begin{align}\label{intro_ConvolutionWithTPStokesFundsol}
\begin{pmatrix}
\uvel \\
\upres
\end{pmatrix} := \fundsoltpstokes * f. 
\end{align}
Having obtained this goal, we shall then examine to which extent regularity such as $\LR{q}$-integrability and pointwise estimates of the solution can be derived from \eqref{intro_ConvolutionWithTPStokesFundsol}.

Since time-periodic data $f:\Rn\times\R\ra\Rn,\ (x,t)\ra f(x,t)$ are non-decaying in $t$, a framework based on classical convolution in $\R^n\times\R$ cannot be applied. Instead, we reformulate \eqref{intro_TPStokesRn} as a system of partial differential equations 
on the locally compact abelian group $\grp:=\R^n\times\R/\per\Z$.
More specifically, we exploit that $\per$-time-periodic functions can naturally be identified with mappings on the 
torus group $\torus:=\R/\per\Z$ in the time variable $t$.
In the setting of the Schwartz-Bruhat space $\SR(\grp)$ and corresponding space of tempered distributions
$\TDR(\grp)$, we can then define a fundamental solution $\fundsoltpstokes$ to \eqref{intro_TPStokesRn} as a 
tensor-field
\begin{align}\label{intro_tpStokesFundForm}
\fundsoltpstokes:=
\begin{pmatrix}
\fundsoltpstokesvel_{11} & \ldots  & \fundsoltpstokesvel_{1n} \\
\vdots & \ddots & \vdots\\
\fundsoltpstokesvel_{n1} & \ldots  & \fundsoltpstokesvel_{nn} \\
\fundsoltpstokespres_{1} & \ldots  & \fundsoltpstokespres_{n} 
\end{pmatrix}
\in \TDR(\grp)^{(n+1)\times n}
\end{align}
that satisfies
\begin{align}\label{intro_tpStokesFundEq}
\begin{pdeq}
&\partial_t\fundsoltpstokesvel_{ij}-\Delta\fundsoltpstokesvel_{ij} + \partial_i \fundsoltpstokespres_j = \delta_{ij}\delta_\grp,  \\
&\partial_i\fundsoltpstokesvel_{ij} =0
\end{pdeq}
\end{align} 
in the sense of $\TDR(\grp)$-distributions. A solution to the time-periodic Stokes system \eqref{intro_TPStokesRn} is then given by \eqref{intro_ConvolutionWithTPStokesFundsol},
provided the convolution is taken over the group $\grp$. 

The aim in the following is to identify a tensor-field $\fundsoltpstokes\in\TDR(\grp)^{(n+1)\times n}$ 
satisfying \eqref{intro_tpStokesFundEq}.
We shall describe $\fundsoltpstokes$ as a sum of the steady-state Stokes
fundamental solution $\fundsolstokes$ and a remainder part satisfying remarkably good integrability and pointwise decay estimates.
It is well-known that the components of the velocity part $\fundsolstokesvel\in\TDR(\R^n)^{n\times n}$ and pressure part $\fundsolstokespres\in\TDR(\R^n)^n$ of $\fundsolstokes$ are 
functions
\begin{align*}
&\fundsolstokesvel_{ij}(x):= 
\begin{pdeq}
&\frac{1}{2\omega_n}\Bp{\delta_{ij}\log\bp{\snorm{x}^{-1}}+\frac{x_ix_j}{\snorm{x}^2} } && \text{if }n=2, \\
&\frac{1}{2\omega_n}\Bp{\delta_{ij}\frac{1}{n-2}\snorm{x}^{2-n}+\frac{x_ix_j}{\snorm{x}^n}} && \text{if }n\geq 3, 
\end{pdeq}\\
&\fundsolstokespres_{i}(x):= 
\frac{1}{\omega_n}\frac{x_i}{\snorm{x}^n},
\end{align*}
respectively; see for example \cite[IV.2]{GaldiBookNew}. Here, $\omega_n$ denotes the surface area of the $(n-1)$-dimensional unit sphere in $\Rn$. 
Our main theorem reads:

\begin{thm}\label{mainthm_IdandEstOfFundsolThm}
Let $n\geq 2$. There is a fundamental solution $\fundsoltpstokes\in\TDR(\grp)^{(n+1)\times n}$ to the time-periodic Stokes equations \eqref{intro_TPStokesRn} on the form \eqref{intro_tpStokesFundForm} satisfying
\eqref{intro_tpStokesFundEq} and 
\begin{align}
&\fundsoltpstokesvel = \fundsolstokesvel\otimes \onedist_{\torus} + \fundsolcompl,\label{mainthm_IdandEstOfFundsolThm_decompVel}\\
&\fundsoltpstokespres = \fundsolstokespres \otimes \delta_{\torus}\label{mainthm_IdandEstOfFundsolThm_decompPres}
\end{align}
with $\fundsolcompl\in\TDR(\grp)^{n\times n}$ satisfying 
\begin{align}
&\forall q\in\Bp{1,\frac{n}{n-1}}:\quad \fundsolcompl\in\LR{q}(\grp)^{n\times n},\label{mainthm_IdandEstOfFundsolThm_ComplSummability}\\
&\forall r\in [1,\infty)\ \forall\epsilon>0\ \exists\Ccn{C}>0\ \forall \snorm{x}\geq \epsilon:\  \norm{\fundsolcompl(x,\cdot)}_{\LR{r}(\torus)} \leq \frac{\Ccn{C}}{\snorm{x}^n},\label{mainthm_IdandEstOfFundsolThm_ComplPointwiseEst}\\
&\forall q\in(1,\infty)\ \exists \Ccn{C}>0\ \forall f\in\SR(\grp)^n:\  \norm{\fundsolcompl*f}_{\WSR{2,1}{q}(\grp)} \leq \Ccn{C}\,\norm{f}_{\LR{q}(\grp)},  \label{mainthm_IdandEstOfFundsolThm_ConvlFundsolcomplLqEst}
\end{align}
where $\torus$ denotes the torus group $\torus:=\R/\per\Z$, $\onedist_\torus\in\TDR(\torus)$ the constant $1$, $\delta_\torus\in\TDR(\torus)$ the Dirac delta distribution on $\torus$, $*$ the convolution on $\grp$, and $\WSR{2,1}{q}(\grp)$ the Sobolev space of order $2$ in $x$ and order $1$ in $t$. 
\end{thm}

\begin{rem}
We shall briefly demonstrate how the fundamental solution \eqref{mainthm_IdandEstOfFundsolThm_decompVel}--\eqref{mainthm_IdandEstOfFundsolThm_decompPres}
can be applied in a more classical setting of the time-periodic Stokes equations to obtain a representation formula, integrability properties and decay estimates of 
a solution. 
The time-periodic Stokes equations are typically studied in a function analytical framework based on the function space 
\begin{align*}
\CRciper(\R^n\times\R) := \setc{f\in\CRi(\R^n\times \R)}{f(x,t+\per)=f(x,t)\ \wedge\ f\in\CRci\bp{\R^n\times[0,\per]}},
\end{align*}
upon which $\norm{f}_q:=\norm{f}_{\LR{q}(\Rn\times[0,\per])}$ is a norm.
Time-periodic Lebesgue and Sobolev spaces are defined as 
\begin{align*}
&\LRper{q}(\Rn\times\R):= \closure{\CRciper(\R^n\times\R)}{\norm{\cdot}_{q}},\\
&\WSRper{2,1}{q}(\Rn\times\R) := \closure{\CRciper(\R^n\times\R)}{\norm{\cdot}_{2,1,q}},\quad 
\norm{f}_{2,1,q}:=\Bp{\sum_{\snorm{\alpha}\leq 2} \norm{\partial_x^\alpha f}_{q}^q + \sum_{\snorm{\beta}\leq 1} \norm{\partial_t^\beta f}_{q}^q}^{\frac{1}{q}}.
\end{align*}
It is easy to see that $\LRper{q}(\Rn\times\R)$ and $\WSRper{2,1}{q}(\Rn\times\R)$ are isometrically isomorphic to 
$\LR{q}(\grp)$ and $\WSR{2,1}{q}(\grp)$, respectively. Regarding $\fundsolcompl$ as a tensor-field in $\LRper{q}(\Rn\times\R)$, we obtain 
by Theorem \ref{mainthm_IdandEstOfFundsolThm} for any sufficiently smooth vector-field $f$, say $f\in\CRciper(\R^n\times\R)^n$, a solution $(\uvel,\upres)$
to \eqref{intro_TPStokesRn} given by $\uvel:=\uvel_1+\uvel_2$ with 
\begin{align}\label{repformular_classical}
\begin{aligned}
&\uvel_1:= \Bb{\fundsolstokesvel *_{\Rn }\Bp{\frac{1}{\per}\int_{0}^\per f(\cdot,s)\,\ds}}(x,t),\\
&\uvel_2:= \int_\Rn \frac{1}{\per}\int_0^\per \fundsolcompl(x-y,t-s)\,f(y,s)\,\ds\dy
\end{aligned}
\end{align}
and $\upres(x,t) := \bb{\fundsolstokespres*_\Rn f(\cdot,t)}(x)$.  
Properties of $\uvel_1$ and $\upres$ can be derived directly from the Stokes fundamental solution $(\fundsolstokesvel,\fundsolstokespres)$,
which, given the simple structure of $(\fundsolstokesvel,\fundsolstokespres)$, is elementary and can be found in standard literature such as \cite{GaldiBookNew} and
\cite{VarnhornBook}.
To fully understand the structure of a time-periodic solution, it therefore remains to investigate $\uvel_2$. For this purpose,
\eqref{mainthm_IdandEstOfFundsolThm_ComplSummability}--\eqref{mainthm_IdandEstOfFundsolThm_ConvlFundsolcomplLqEst} of Theorem \ref{mainthm_IdandEstOfFundsolThm} is useful. For example, \eqref{mainthm_IdandEstOfFundsolThm_ConvlFundsolcomplLqEst} yields
integrability $\uvel_2\in\WSRper{2,1}{q}(\Rn\times\R)$, and from \eqref{mainthm_IdandEstOfFundsolThm_ComplPointwiseEst} the pointwise decay estimate
$\snorm{\uvel_2(x,t)}\leq\Ccn{C} \snorm{x}^{-n}$ can be derived for large values of $x$.
\end{rem}

\begin{rem}\label{remOnAsymp}
Theorem \ref{mainthm_IdandEstOfFundsolThm} implies that $\fundsolcompl$ decays faster than $\fundsolstokesvel$ as $\snorm{x}\ra\infty$; both in terms of summability \eqref{mainthm_IdandEstOfFundsolThm_ComplSummability} and pointwise \eqref{mainthm_IdandEstOfFundsolThm_ComplPointwiseEst}.
This information provides us with a valuable insight into the asymptotic structure as $\snorm{x}\ra\infty$ of a time-periodic solution to the Stokes equations.
More precisely, from the representation formula $\uvel=\uvel_1+\uvel_2$ 
with $\uvel_1$ and $\uvel_2$ given by \eqref{repformular_classical}, 
and the fact that $\fundsolcompl$ decays faster than $\fundsolstokesvel$ as $\snorm{x}\ra\infty$,
it follows that the leading term in an asymptotic expansion of $\uvel$ coincides with 
the leading term in the expansion of $\uvel_1$. 
Since $\uvel_1$ is a solution to a steady-state Stokes problem, it is well-known how to identify its leading term. 
In conclusion, Theorem \ref{mainthm_IdandEstOfFundsolThm} tells us that 
time-periodic solutions to the Stokes equations essentially have the same well-known asymptotic structure as $\snorm{x}\ra\infty$ as steady-state solutions---a nontrivial fact, which is not clear at the outset.
\end{rem}

The Stokes system is a linearization of the nonlinear Navier-Stokes system. A fundamental solution to the time-periodic Stokes equations 
can therefore be used to develop a linear theory for the time-periodic Navier-Stokes problem.
The study of the time-periodic Navier-Stokes equations was initiated by \textsc{Serrin} \cite{Serrin_PeriodicSolutionsNS1959},
\textsc{Prodi} \cite{Prodi1960}, and \textsc{Yudovich} \cite{Yudovich60}. Since then, a number of papers have appeared based on the 
techniques proposed by these authors. The methods all have in common that the time-periodic problem is investigated in a setting of the corresponding initial-value problem, and time-periodicity of a solution only established a posterior. With an appropriate time-periodic linear theory, a more direct approach to the time-periodic Navier-Stokes problem can be developed, which may reveal more information on the solutions. The asymptotic structure mentioned in Remark \ref{remOnAsymp} is but one example.

\section{Preliminaries}

Points in $\R^n\times\R$ are denoted by $(x,t)$ with $x\in\Rn$ and $t\in\R$.
We refer to $x$ as the spatial and to $t$ as the time variable. 

We denote by $\B_R:=\B_R(0)$ balls in $\Rn$ centered at $0$. Moreover, we let $\B_{R,r}:=\B_R\setminus\overline{\B_r}$ and $\B^R:=\Rn\setminus\overline{\B_r}$

For a sufficiently regular function $u:\Rn\times\R\ra\R$, we put $\partial_i u:=\partial_{x_i} u$.
The differential operators $\Delta$, $\grad$ and $\Div$ act only in the spatial variables. For example, 
$\Div u:=\sum_{j=1}^n\partial_j u_j$ denotes the divergence of $u$ with respect to the $x$ variables.

We let $\grp$ denote the group $\grp:=\R^n\times\torus$, with $\torus$ denoting the torus group $\torus:=\R/\per\Z$. 
$\grp$ is equipped with the quotient topology and differentiable structure inherited from $\Rn\times\R$ via the quotient mapping
$\quotientmap:\Rn\times\R\ra\grp$, $\quotientmap(x,t):=\bp{x,[t]}$.
Clearly, $\grp$ is a locally compact abelian group with Haar measure given by the product of the Lebesgue measure $\dx$ on $\Rn$ and the (normalized) Haar measure $\dt$ on $\torus$. 
We implicitly identify $\torus$ with the interval $[0,\per)$, whence the (normalized) Haar measure on $\torus$ is determined by
\begin{align*}
\forall f\in\CR{}(\torus):\quad \int_\torus f\,\dt := \frac{1}{\per}\int_0^\per f(t)\,\dt.
\end{align*}
We identify the dual group $\dualgrp$ with $\Rn\times\Z$ and denote points in $\dualgrp$ by $(\xi,k)$. 

We denote by $\SR(\grp)$ the Schwartz-Bruhat space of generalized Schwartz functions; see \cite{Bruhat61}. By 
$\TDR(\grp)$ we denote the corresponding space of tempered distributions. The Fourier transform on $\grp$ and its inverse takes the form 
\begin{align*}
&\FT_\grp:\SR(\grp)\ra\SR(\dualgrp),\quad \FT_\grp\nb{\uvel}(\xi,k):=
\int_{\R^n}\int_\torus \uvel(x,t)\,\e^{-ix\cdot\xi-ik\perf t}\,\dt\dx,\\
&\iFT_\grp:\SR(\dualgrp)\ra\SR(\grp),\quad \iFT\nb{\wvel}(x,t):=
\sum_{k\in\Z}\,\int_{\R^n} \wvel(\xi,k)\,\e^{ix\cdot\xi+ik\perf t}\,\dxi,
\end{align*}
respectively, provided the Lebesgue measure $\dxi$ is normalized appropriately. By duality, $\FT_\grp$ extends to a homeomorphism $\FT_\grp:\TDR(\grp)\ra\TDR(\dualgrp)$. Observe that $\FT_\grp=\FT_\Rn\circ\FT_\torus$.

We denote by $\delta_\Rn$, $\delta_\torus$, $\delta_\Z$ the Dirac delta distribution on $\Rn$, $\torus$ and $\Z$, respectively.
Observe that $\delta_\Z$ is a function with $\delta_\Z(k)=1$ if $k=0$ and $\delta_\Z(k)=0$ otherwise. Also note that 
$\FT_\torus\nb{1_\torus}=\delta_\Z$.

Given a tensor $\Gamma\in\TDR(\grp)^{n\times m}$, we define the convolution of $\Gamma$ with vector field $f\in\SR(\grp)^m$ as 
the vector field $\Gamma * f \in\TDR(\grp)^n$ with $\nb{\Gamma*f}_{i} := \Gamma_{ij}*f_j$.

The $\LR{q}(\grp)$-spaces with norm $\norm{\cdot}_q$ are defined in the usual way via the Haar measure $\dx\dt$ on $\grp$. We further introduce 
the Sobolev space
\begin{align*}
&\WSR{2,1}{q}(\grp) := \closure{\CRci(\grp)}{\norm{\cdot}_{2,1,q}},\quad 
\norm{f}_{2,1,q}:=\Bp{\sum_{\snorm{\alpha}\leq 2} \norm{\partial_x^\alpha f}_{q}^q + \sum_{\snorm{\beta}\leq 1} \norm{\partial_t^\beta f}_{q}^q}^{\frac{1}{q}},
\end{align*}
where $\CRci(\grp)$ denotes the space of smooth functions of compact support on $\grp$.

We emphasize at this point that a framework based on $\grp$ is a natural setting for the time-period Stokes equations. It it easy to see 
that lifting by the restriction $\quotientmap_{|\R^n\times[0,\per)}$ of the quotient mapping provides us with an equivalence between the time-periodic 
Stokes problem \eqref{intro_TPStokesRn} and the system
\begin{align*}
\begin{pdeq}
&\partial_t\uvel-\Delta\uvel + \grad\upres = f && \tin\grp, \\
&\Div\uvel =0 && \tin\grp.
\end{pdeq}
\end{align*}
An immediate advantage obtained by writing the time-periodic Stokes problem as system of equations on $\grp$ is the ability to then apply the Fourier transform $\FT_\grp$ and re-write the problem in terms of Fourier symbols. We shall take advantage of this possibility in the proof of the main theorem below.  

We use the symbol $C$ for all constants. In particular, $C$ may represent different constants in the scope of a proof.

\section{Proof of main theorem}

\begin{proof}[Proof of Theorem \ref{mainthm_IdandEstOfFundsolThm}]\newCCtr[c]{proofmainthmConst}

Put 
\begin{align}\label{proofmainthm_defoffundsolcompl}
\fundsolcompl := \iFT_\grp\Bb{\frac{1-\projsymbol(k)}{\snorm{\xi}^2+i\perf k}\Bp{\idmatrix-\frac{\xi\otimes\xi}{\snorm{\xi}^2}}},
\end{align}
where $\idmatrix\in\R^{n\times n}$ denotes the identity matrix.
Since 
\begin{align}\label{proofmainthm_defofmultiplier}
\Mmultiplier:\dualgrp\ra\C,\quad \Mmultiplier(\xi,k):=\frac{1-\projsymbol(k)}{\snorm{\xi}^2+i\perf k}
\end{align}
is bounded, that is, $\Mmultiplier\in\LR{\infty}(\dualgrp)$, the inverse Fourier transform in \eqref{proofmainthm_defoffundsolcompl} 
is well-defined as a distribution in $\TDR(\grp)^{n\times n}$. 
Now define $\fundsoltpstokesvel$ and $\fundsoltpstokespres$ as in \eqref{mainthm_IdandEstOfFundsolThm_decompVel} and \eqref{mainthm_IdandEstOfFundsolThm_decompPres}. 
It is then easy to verify that $(\fundsoltpstokesvel,\fundsoltpstokespres)$ is a solution to \eqref{intro_tpStokesFundEq}. 

It remains to 
show \eqref{mainthm_IdandEstOfFundsolThm_ComplSummability}--\eqref{mainthm_IdandEstOfFundsolThm_ConvlFundsolcomplLqEst}.
For this purpose, we introduce for $k\in\Z\setminus\set{0}$ the function
\begin{align}\label{defofFundsolSSRk}
\fundsolssrk:\R^n\setminus\set{0}\ra\C,\quad \fundsolssrk(x) := \frac{i}{4} \BBp{\frac{\sqrt{-i\perf k}}{2\pi \snorm{x}}\,}^{\frac{n-2}{2}}\hankelone{\frac{n}{2}-1}
\Bp{\sqrt{-i\perf k}\cdot \snorm{x}},
\end{align}
where $\hankelone{\alpha}$ denotes the Hankel function of the first kind, and $\sqrt{z}$ the square root of $z$ with \emph{positive} imaginary part. As one readily verifies, 
$\fundsolssrk$ is a fundamental solution to the Helmholtz equation
\begin{align}\label{HelmholtzEqFundsolEq}
\Bp{-\Delta + i\perf k} \fundsolssrk = \delta_{\R^n}\quad \tin\R^n.
\end{align}
Clearly, $\fundsolssrk\in\TDR(\R^n)$. Moreover, its Fourier transform is given by the function
\begin{align}\label{fundsolssrkFT}
\FT_{\R^n}\bb{\fundsolssrk}(\xi) = \frac{1}{\snorm{\xi}^2+i\perf k}.
\end{align}
From the estimates in Lemma \ref{l2estFundsolssrk} below, we see that
\begin{align*}
\int_{\R^n}\Bp{\sum_{k\in\Z\setminus\set{0}} \snormL{{\fundsolssrk}}^2}^\frac{q}{2}\dx < \infty
\end{align*}
for $q\in\bp{1,\frac{n}{n-1}}$. By H\"older's inequality and Parseval's theorem, we thus deduce
\begin{align*}
&\int_\torus \int_{\R^n} \snorml{\iFT_{\torus}\bb{\bp{1-\projsymbol(k)}\cdot\fundsolssrk(x)}(t)}^q\,\dx\dt\\ 
&\qquad\qquad\leq \Ccn{C}\, \int_{\R^n}\Bp{\int_\torus  \snormL{\iFT_{\torus}\bb{\bp{1-\projsymbol(k)}\cdot\fundsolssrk}}^2\,\dt}^\frac{q}{2}\dx\\
&\qquad\qquad\leq \Ccn{C}\, \int_{\R^n}\Bp{\sum_{k\in\Z\setminus\set{0}} \snormL{{\fundsolssrk}}^2}^\frac{q}{2}\dx<\infty.
\end{align*}
It is well-known that the Riesz transform $\riesztrans_k(f):=\iFT_{\R^n}\bb{\frac{\xi_k}{\snorm{\xi}}\cdot\FT_{\R^n}\nb{f}}$ is bounded on $\LR{q}(\R^n)$ for all $q\in (1,\infty)$. Consequently, we obtain 
$\riesztrans_i\circ \riesztrans_j \bp{\iFT_{\torus}\bb{\bp{1-\projsymbol(k)}\cdot\fundsolssrk}}\in\LR{q}(\grp)$ for $q\in\bp{1,\frac{n}{n-1}}$. Recalling \eqref{fundsolssrkFT}, we compute
\begin{align*}
\bb{\delta_{ij}\riesztrans_h\circ\riesztrans_h-\riesztrans_i\circ\riesztrans_j} \bp{\iFT_{\torus}\bb{\bp{1-\projsymbol(k)}\cdot{\fundsolssrk}}}
=\fundsolcompl_{ij}
\end{align*}
and conclude \eqref{mainthm_IdandEstOfFundsolThm_ComplSummability}.

In order to show \eqref{mainthm_IdandEstOfFundsolThm_ComplPointwiseEst}, we further introduce
\begin{align*}
\fundsollaplace:\R^n\setminus\set{0} \ra \C,\quad \fundsollaplace:=
\begin{pdeq}
&-\frac{1}{2\pi} \log\snorm{x} &&(n=2),\\
&\frac{1}{(n-2)\omega_n} \snorm{x}^{2-n} &&(n>2),
\end{pdeq}
\end{align*}
which is the fundamental solution to the Laplace equation $\Delta\fundsollaplace = \delta_{\R^n}$ in $\R^n$.
As one may verify directly from the pointwise definitions of $\fundsolssrk$ and $\fundsollaplace$, the 
convolution integral 
\begin{align}\label{defoffundsollaplaceCONVLfundsolssrk}
\int_{\R^n} \fundsollaplace(x-y)\,\fundsolssrk(y)\,\dy =: \fundsollaplace*\fundsolssrk (x)
\end{align}
exists for all $x\in\R^n\setminus\set{0}$. In fact, the function given by $\fundsollaplace*\fundsolssrk$ belongs to $\LRloc{1}(\R^n)$
and defines a tempered distribution in $\TDR(\R^n)$. One may further verify that also the second order derivatives of $\fundsollaplace*\fundsolssrk$ are given by convolution integrals
\begin{align}\label{fundsollaplaceCONVLfundsolssrk_secondorderderivatives}
\partial_i\partial_j \nb{\fundsollaplace*\fundsolssrk} (x) = \int_{\R^n} \partial_i\fundsollaplace(x-y)\,\partial_j\fundsolssrk(y)\,\dy, 
\end{align}
from which it follows that their Fourier transform are functions
\begin{align*}
\FT_{\R^n}\bb{\partial_i\partial_j \nb{\fundsollaplace*\fundsolssrk}}(\xi) = \frac{\xi_i\xi_j}{\snorm{\xi}^2}\frac{1}{\snorm{\xi}^2+i\perf k}.
\end{align*}
We infer from the expression above that 
\begin{align*}
\fundsolcompl_{ij} = \iFT_{\torus}\bb{\bp{1-\projsymbol(k)}\cdot \bb{\delta_{ij}\partial_h\partial_h-\partial_i\partial_j}\nb{\fundsollaplace*\fundsolssrk}}.
\end{align*}
Employing Hausdorff-Young's inequality in combination with the pointwise estimate from Lemma \ref{pointwiseEstLem} below, we obtain for $r\in[2,\infty)$
\begin{align*}
\norm{\fundsolcompl(x,\cdot)}_{\LR{r}(\torus)} &\leq \Bp{\sum_{k\in\Z} 
\snormL{\bp{1-\projsymbol(k)}\cdot \bb{\delta_{ij}\partial_h\partial_h-\partial_i\partial_j}\nb{\fundsollaplace*\fundsolssrk}(x)}^{r^*} }^{\frac{1}{r^*}}\\
&\leq \Ccn{C}\,\snorm{x}^{-n}\, \Bp{\sum_{k\in\Z\setminus\set{0}} \snorm{k}^{-r^*} }^{\frac{1}{r^*}} \leq \Ccn{C}\,\snorm{x}^{-n},
\end{align*}
which concludes \eqref{mainthm_IdandEstOfFundsolThm_ComplPointwiseEst}.

The convolution $\fundsolcompl*f$ can be expressed in terms of a Fourier multiplier
\begin{align*}
\fundsolcompl*f = \iFT_\grp\Bb{\Mmultiplier(\xi,k)\Bp{\idmatrix-\frac{\xi\otimes\xi}{\snorm{\xi}^2}} \FT_\grp\nb{f}},
\end{align*}
with $M$ given by \eqref{proofmainthm_defofmultiplier}. As already mentioned, $M\in\LR{\infty}(\dualgrp)$. As one may verify, also 
second order spatial derivatives $\partial_i\partial_j M\in\LR{\infty}(\dualgrp)$ and the time derivative $\partial_t M\in\LR{\infty}(\dualgrp)$ are bounded.
Based on this information, \eqref{mainthm_IdandEstOfFundsolThm_ConvlFundsolcomplLqEst} can be established. For the details of the argument, we refer
the reader to \cite[Proof of Theorem 4.8]{mrtpns}.
\end{proof}

\begin{lem}\label{l2estFundsolssrk}
The function $\fundsolssrk$ defined in \eqref{defofFundsolSSRk} satisfies
\begin{align}
\Bp{\sum_{k\in\Z\setminus\set{0}} \snorml{\fundsolssrk(x)}^2}^\frac{1}{2} \leq \Ccn{C} \, \snorm{x}^{1-n}\,\e^{-\half\sqrt{\frac{\pi}{\per}}\snorm{x}}.\label{l2estFundsolssrkngeq}
\end{align}
\end{lem}
\begin{proof}
The estimates are based on the asymptotic properties of Hankel functions summarized in Lemma \ref{estHankelFktlem} below.
We start with the case $n>2$.
Employing \eqref{estHankelFktAtInfty} with $\epsilon=1$, we deduce
\begin{align}
&\forall k\in\Z\ \forall \snorm{x}\geq \sqrt{\frac{\per}{2\pi}}:\quad 
\snormL{\hankelone{\frac{n}{2}-1}\Bp{\sqrt{-i\perf k}\cdot \snorm{x}}} \leq \Ccn{C}\,\snorm{k}^{-\frac{1}{4}}\,\snorm{x}^{-\half}\,
\e^{-\sqrt{\frac{\pi}{\per}}\snorm{k}^\half \snorm{x}}\label{oldestHankelFktAtInfty}.
\end{align}
Employing \eqref{estHankelFktAtZero} with $R=1$, we obtain:
\begin{align}
&\forall k\in\Z\ \forall \snorm{x}\leq \sqrt{\iperf}\snorm{k}^{-\half}:\quad
\snormL{\hankelone{\frac{n}{2}-1}\Bp{\sqrt{-i\perf k}\cdot \snorm{x}}} \leq \Ccn{C}\,\snorm{k}^{-\frac{n-2}{4}}\,\snorm{x}^{-\frac{n-2}{2}}.
\label{oldestHankelFktAtZero}
\end{align}
It follows that
\begin{align}
\begin{aligned}
\sum_{k\in\Z\setminus\set{0}} \snorml{\fundsolssrk(x)}^2 
&\leq \Ccn{C} 
\Bp{\sum_{\snorm{k}\leq \iperf\snorm{x}^{-2}} \snorm{k}^{\frac{n-2}{2}}\, \snorm{x}^{2-n}\, \snorm{k}^{\frac{2-n}{2}}\,\snorm{x}^{2-n}\\
&\qquad + \sum_{\snorm{k}> \iperf\snorm{x}^{-2}} \snorm{k}^{\frac{n-2}{2}}\, \snorm{x}^{2-n}\, \snorm{k}^{-\half}\, \snorm{x}^{-1} 
\e^{-2\sqrt{\frac{\pi}{\per}}\snorm{k}^\half\snorm{x}}}\\
&\leq \Ccn{C} \Bp{
\snorm{x}^{-2}\cdot\snorm{x}^{2(2-n)}\cdot \chi_{\bb{0,\sqrt{\iperf}}}(\snorm{x}) \\
&\qquad +  \sum_{\snorm{k}\geq 1 } \snorm{k}^{\frac{n-3}{2}}\, \snorm{x}^{1-n}\, \e^{-2\sqrt{\frac{\pi}{\per}}\snorm{k}^\half\snorm{x}}}.
\end{aligned}\label{l2estFundsolssrk_iterationpoint}
\end{align}
For $\snorm{q}<1$ we observe that
\begin{align*}
\sum_{k\geq 1 } \snorm{k}^{\frac{n-3}{2}}\, q^{k^\half} 
&= \sum_{j=1}^\infty \sum_{k=j^2}^{(j+1)^2-1} k^{\frac{n-3}{2}}\, q^{k^\half}\\
&\leq \sum_{j=1}^\infty \sum_{k=j^2}^{(j+1)^2-1} \np{j+1}^{n-3}\, q^{j}\\
&= \sum_{j=1}^\infty j\, \np{j+1}^{n-3}\, q^{j}\\
&\leq q\,\sum_{j=1}^\infty j\, (j+1)\,(j+2)\ldots\,(j+n-3)\, q^{j-1}\\
&= q \, \partial_q^{n-2}\Bb{\sum_{j=1}^\infty q^{j+n-3}}
= q \, \partial_q^{n-2}\bb{(1-q)^{-1}} \\
&= (n-2)!\cdot \,q\,(1-q)^{1-n},
\end{align*}
from which we deduce 
\begin{align*}
\sum_{k\in\Z\setminus\set{0}} \snorml{\fundsolssrk(x)}^2 
&\leq \Ccn{C} \Bp{
\snorm{x}^{2(1-n)}\cdot \chi_{\bb{0,\sqrt{\iperf}}}(\snorm{x})   \\
&\qquad  + \snorm{x}^{1-n}\,\e^{-2\sqrt{\frac{\pi}{\per}}\snorm{x}}\bp{1-\e^{-2\sqrt{\frac{\pi}{\per}}\snorm{x}}}^{1-n}}\\
&\leq \Ccn{C} \snorm{x}^{2(1-n)}\cdot \e^{-\sqrt{\frac{\pi}{\per}}\snorm{x}}
\end{align*}
and consequently \eqref{l2estFundsolssrkngeq} in the case $n>2$. In the case $n=2$, 
we employ \eqref{estHankelzeroFktAtZero} to deduce
\begin{align}
&\forall k\in\Z\ \forall \snorm{x}\leq \sqrt{\iperf}\snorm{k}^{-\half}:\quad
\snormL{\hankelone{0}\Bp{\sqrt{-i\perf k}\cdot \snorm{x}}} \leq \Ccn{C}\,\snormL{\log\Bp{\sqrt{\perf} \snorm{k}^\half \snorm{x}}}.
\label{oldestHankelzeroFktAtZero}
\end{align}
It follows in the case $n=2$ that\footnote{I would like to thank Prof.~Toshiaki Hishida for suggesting this estimate to me and thereby improving my original proof.}
\begin{align}\label{HishidaEst}
\begin{aligned}
\sum_{\snorm{k}\leq \iperf\snorm{x}^{-2}} \snorml{\fundsolssrk(x)}^2 &\leq 
\Ccn{C} \sum_{\snorm{k}\leq \iperf\snorm{x}^{-2}} \snormL{\log\Bp{\sqrt{\perf} \snorm{k}^\half \snorm{x}}}^2\\
&\leq \Ccn{C} \int_0^{\iperf\snorm{x}^{-2}} \snormL{\log\Bp{\sqrt{\perf} t^\half \snorm{x}}}^2\,\dt\,\cdot \chi_{\bb{0,\sqrt{\iperf}}}(\snorm{x})\\
&\leq \Ccn{C} \snorm{x}^{-2} \int_0^1 \snorml{\log(s)}^2\,s\,\ds \cdot \chi_{\bb{0,\sqrt{\iperf}}}(\snorm{x})\\
&\leq \Ccn{C} \snorm{x}^{-2} \cdot \chi_{\bb{0,\sqrt{\iperf}}}(\snorm{x}).
\end{aligned}
\end{align}
Estimate \eqref{oldestHankelFktAtInfty} is still valid in the case $n=2$. We can thus proceed as in \eqref{l2estFundsolssrk_iterationpoint}
and obtain \eqref{l2estFundsolssrkngeq} also in the case $n=2$. 
\end{proof}

\begin{lem}\label{pointwiseEstLem}
The convolution $\fundsollaplace*\fundsolssrk$ defined in \eqref{defoffundsollaplaceCONVLfundsolssrk} satisfies
\begin{align}
&\forall \epsilon>0\ \exists \Ccn{C}>0\ \forall \snorm{x}\geq \epsilon:\ \snorml{\partial_i\partial_j\nb{\fundsollaplace*\fundsolssrk}(x)} \leq \Ccn{C}\,\snorm{k}^{-1}\,\snorm{x}^{-n}.\label{pointwiseEstLemEst}
\end{align}
\end{lem}
\begin{proof}
Fix $\epsilon>0$ and consider some $x\in\R^n$ with $\snorm{x}\geq \epsilon$. Put $R:=\frac{\snorm{x}}{2}$.  
Let $\cutoff\in\CRci(\R;\R)$ be a ``cut-off'' function with
\begin{align*}
\cutoff(r)=
\begin{pdeq}
&0 && \text{when }{0\leq \snorm{r}\leq \half},\\
&1 && \text{when }{1\leq \snorm{r}\leq 3},\\
&0 && \text{when }{4\leq \snorm{r}}.
\end{pdeq}
\end{align*}
Define $\cutoff_R:\R^n\ra\R$ by $\cutoff_R(y):=\cutoff\bp{R^{-1}\snorm{y}}$. We use $\cutoff_R$ to decompose the integral in \eqref{fundsollaplaceCONVLfundsolssrk_secondorderderivatives} as
\begin{align*}
\partial_i\partial_j \nb{\fundsollaplace*\fundsolssrk} (x) 
&= \int_{B_{4R,R/2}} \partial_i\fundsollaplace(x-y)\,\partial_j\fundsolssrk(y)\,\cutoff_R(y)\,\dy\\ 
&\quad +\int_{B_{R}} \partial_i\fundsollaplace(x-y)\,\partial_j\fundsolssrk(y)\,\bp{1-\cutoff_R(y)}\,\dy\\ 
&\quad +\int_{B^{3R}} \partial_i\fundsollaplace(x-y)\,\partial_j\fundsolssrk(y)\,\bp{1-\cutoff_R(y)}\,\dy\\
&=: I_1(x) + I_2(x) + I_3(x).
\end{align*}
Recalling the definition \eqref{defofFundsolSSRk} of $\fundsolssrk$ as well as the property \eqref{diffOfHankelFormula}
and the estimate \eqref{estHankelFktAtInfty} of the Hankel function, we can estimate for $\snorm{y}\geq R/2$:  
\begin{align*}
\snorml{\partial_j \fundsolssrk(y)} 
&\leq \Ccn{C} \snorm{k}^{\frac{n-2}{4}}\Bp{
\snormL{\partial_j\Bb{\snorm{y}^{\frac{2-n}{2}}}\, \hankelone{\frac{n-2}{2}} \Bp{\sqrt{-i\perf k}\cdot \snorm{y}}}\\
&\qquad\qquad\quad+\snormL{{\snorm{y}^{\frac{2-n}{2}}}\, \partial_j\Bb{\hankelone{\frac{n-2}{2}} \Bp{\sqrt{-i\perf k}\cdot \snorm{y}}}}} \\
&\leq \Ccn{C} \Bp{\snorm{k}^{\frac{n}{4}-\frac{3}{4}}\, \snorm{y}^{-\frac{n}{2}-\half} + \snorm{k}^{\frac{n}{4}-\frac{1}{4}}\, \snorm{y}^{-\frac{n}{2}+\half}}
\e^{-\sqrt{\frac{\pi}{\per}}\snorm{k}^\half\snorm{y}}\\
&\leq \Ccn{C} \snorm{k}^{-1}\,\snorm{y}^{-(n+1)}.
\end{align*}
Consequently, we obtain:
\begin{align*}
\snorml{I_1(x)} &\leq \Ccn{C} \int_{B_{4R,R/2}} \snorm{x-y}^{1-n}\, \snorm{k}^{-1}\,\snorm{y}^{-(n+1)}\,\dy \leq \Ccn{C}\, \snorm{k}^{-1}\,R^{-n}.
\end{align*}
To estimate $I_2$, we integrate partially and employ polar coordinates to deduce
\begin{align*}
\snorml{I_2(x)} 
&\leq \Ccn{C} 
\int_{B_R}\snorml{\partial_j\partial_i \fundsollaplace(x-y)}\, \snorml{\fundsolssrk(y)} +
\snorml{\partial_i\fundsollaplace(x-y)}\, \snorml{\fundsolssrk(y)}\, R^{-1}\,\dy \\
&\leq \Ccn{C} \int_{B_R} R^{-n}\, \snorml{\fundsolssrk(y)}\,\dy\\
&\leq \Ccn{C} \int_{B_R} R^{-n}\, \snorm{k}^{\frac{n-2}{4}}\, \snorm{y}^{\frac{2-n}{2}}\, \snormL{\hankelone{\frac{n-2}{2}}\Bp{\sqrt{-i\perf k}\cdot \snorm{y}}}\,\dy\\
&\leq \Ccn{C} \int_0^R R^{-n}\, \snorm{k}^{\frac{n-2}{4}}\, r^{\frac{n}{2}}\, \snormL{\hankelone{\frac{n-2}{2}}\Bp{\sqrt{-i\perf k}\cdot r}}\,\dr\\
&\leq \Ccn{C} \int_0^\infty R^{-n}\, \snorm{k}^{-1}\, s^{\frac{n}{2}}\, \snormL{\hankelone{\frac{n-2}{2}}\Bp{\sqrt{-i\perf}\cdot\sqrt{\frac{k}{\snorm{k}}}\cdot s}}\,\ds.
\end{align*}
Employing in the case $n>2$ estimate \eqref{estHankelFktAtZero} in combination with \eqref{estHankelFktAtInfty}, we obtain
\begin{align*}
\snorml{I_2(x)} 
&\leq \Ccn{C}\,  R^{-n}\, \snorm{k}^{-1}\, \Bp{ \int_0^1 s^{\frac{n}{2}}\, s^{\frac{2-n}{2}}\,\ds + 
\int_1^\infty s^{\frac{n}{2}}\, s^{-\half}\,\e^{-\sqrt{\frac{\pi}{\per}}s}\,\ds  } \leq \Ccn{C}\,  R^{-n}\, \snorm{k}^{-1}. 
\end{align*}
When $n=2$, we use estimate \eqref{estHankelzeroFktAtZero} in combination with \eqref{estHankelFktAtInfty} and obtain also in this case
\begin{align*}
\snorml{I_2(x)} 
&\leq \Ccn{C}\,  R^{-n}\, \snorm{k}^{-1}\, \Bp{ \int_0^1 s\cdot \snormL{\log\Bp{\sqrt{\frac{\pi}{\per}}s}}\,\ds + 
\int_1^\infty s^{\half}\,\e^{-\sqrt{\frac{\pi}{\per}}s}\,\ds  } \leq \Ccn{C}\,  R^{-n}\, \snorm{k}^{-1}. 
\end{align*}
In order to estimate $I_3$, we again integrate partially and utilize \eqref{estHankelFktAtInfty}:
\begin{align*}
\snorml{I_3(x)} 
&\leq \Ccn{C}\, \int_{B^{3R}} \snorml{\partial_j\partial_i \fundsollaplace(x-y)}\, \snorml{\fundsolssrk(y)} +
\snorml{\partial_i\fundsollaplace(x-y)}\, \snorml{\fundsolssrk(y)}\, R^{-1}\,\dy \\
&\leq \Ccn{C} \int_{B^{3R}} R^{-n}\, \snorml{\fundsolssrk(y)}\,\dy\\
&\leq \Ccn{C} \int_{B^{3R}} R^{-n}\, \snorm{k}^{\frac{n-2}{4}}\, \snorm{y}^{\frac{2-n}{2}}\, \snormL{\hankelone{\frac{n-2}{2}}\Bp{\sqrt{-i\perf k}\cdot \snorm{y}}}\,\dy\\
&\leq \Ccn{C} \int_{B^{3R}} R^{-n}\, \snorm{k}^{\frac{n-3}{4}}\, \snorm{y}^{\frac{1-n}{2}}\,\e^{-\sqrt{\frac{\pi}{\per}}\snorm{k}^\half\snorm{y}} \,\dy\\
&\leq \Ccn{C} \int_{B^{3R}} R^{-n}\, \snorm{k}^{\frac{n-3}{4}}\, \snorm{y}^{\frac{1-n}{2}}\, \bp{\snorm{k}^\half\,\snorm{y}}^{-\frac{n+3}{2}} \,\dy
\leq \Ccn{C}\, R^{-n}\,\snorm{k}^{-\frac{3}{2}}\leq \Ccn{C}\, R^{-n}\,\snorm{k}^{-1}.
\end{align*}
Since $\snorm{x}=2R$, we conclude \eqref{pointwiseEstLemEst} by collecting the estimates for $I_1$, $I_2$ and $I_3$.
\end{proof}

\begin{lem}\label{estHankelFktlem}
Hankel functions are analytic in $\C\setminus\set{0}$ with
\begin{align}\label{diffOfHankelFormula}
\forall\nu\in\C\ \forall z\in\C\setminus\set{0}:\quad \ddz\hankelone{\nu} (z) = \hankelone{\nu-1}(z) - \frac{\nu}{z}\,\hankelone{\nu}(z).
\end{align}
The Hankel functions satisfy the following estimates:
\begin{align}
&\forall\nu\in\C\ \forall\epsilon>0\ \exists \Ccn[estHankelFktInftyConst]{C}>0\ \forall \snorm{z}\geq \epsilon:&&
\snorml{\hankelone{\nu} \np{z}}\leq \const{estHankelFktInftyConst}\, \snorm{z}^{-\frac{1}{2}}\, \e^{-\impart z},
\label{estHankelFktAtInfty}\\
&\forall\nu\in\R_+\ \forall R>0\ \exists \Ccn[estHankelFktZeroConst]{C}>0\ \forall \snorm{z}\leq R:&&
\snorml{\hankelone{\nu} \np{z}}\leq 
\const{estHankelFktZeroConst}\, \snorm{z}^{-\nu},\label{estHankelFktAtZero}\\
&\forall R>0\ \exists \Ccn[estHankelzeroFktZeroConst]{C}>0\ \forall \snorm{z}\leq R: && 
\snorml{\hankelone{0} \np{z}}\leq 
\const{estHankelzeroFktZeroConst}\, \snorml{\log\np{\snorm{z}}}. \label{estHankelzeroFktAtZero}
\end{align}
\end{lem}
\begin{proof}
The recurrence relation \eqref{diffOfHankelFormula} is a well-know property of various Bessel functions; see for example
\cite[9.1.27]{AbramowitzStegunHandbook}.
We refer to \cite[9.2.3]{AbramowitzStegunHandbook} for the asymptotic behaviour \eqref{estHankelFktAtInfty} of $\hankelone{\nu}(z)$ as $z\ra \infty$.
See \cite[9.1.9 and 9.1.8]{AbramowitzStegunHandbook} for the asymptotic behaviour \eqref{estHankelFktAtZero} and \eqref{estHankelzeroFktAtZero} 
of $\hankelone{\nu}(z)$ as $z\ra 0$.
\end{proof}

\bibliographystyle{abbrv}

\begin{thebibliography}{1}

\bibitem{AbramowitzStegunHandbook}
M.~{Abramowitz} and I.~A. {Stegun}, editors.
\newblock {\em {Handbook of mathematical functions with formulas, graphs, and
  mathematical tables. 10th printing, with corrections.}}
\newblock {New York: John Wiley \& Sons}, 1972.

\bibitem{Bruhat61}
F.~Bruhat.
\newblock {Distributions sur un groupe localement compact et applications \`a
  l'\'etude des repr\'esentations des groupes $p$-adiques.}
\newblock {\em Bull. Soc. Math. Fr.}, 89:43--75, 1961.

\bibitem{GaldiBookNew}
G.~P. Galdi.
\newblock {\em {An introduction to the mathematical theory of the Navier-Stokes
  equations. Steady-state problems. 2nd ed.}}
\newblock {New York: Springer}, 2011.

\bibitem{mrtpns}
M.~Kyed.
\newblock {Maximal regularity of the time-periodic linearized Navier-Stokes
  system}.
\newblock {\em {J. Math. Fluid Mech.}}, 16(3):523--538, 2014.

\bibitem{Prodi1960}
G.~Prodi.
\newblock Qualche risultato riguardo alle equazioni di {N}avier-{S}tokes nel
  caso bidimensionale.
\newblock {\em Rend. Sem. Mat. Univ. Padova}, 30:1--15, 1960.
                                                                                                                                                                                                                                             
\bibitem{Serrin_PeriodicSolutionsNS1959}                                                                                                                                                                                                     
J.~Serrin.                                                                                                                                                                                                                                   
\newblock A note on the existence of periodic solutions of the                                                                                                                                                                               
  {N}avier-{S}tokes equations.                                                                                                                                                                                                               
\newblock {\em Arch. Rational Mech. Anal.}, 3:120--122, 1959.                                                                                                                                                                                
                                                                                                                                                                                                                                             
\bibitem{VarnhornBook}                                                                                                                                                                                                                       
W.~{Varnhorn}.                                                                                                                                                                                                                               
\newblock {\em {The Stokes equations.}}                                                                                                                                                                                                      
\newblock Berlin: Akademie Verlag, 1994.                                                                                                                                                                                                     
                                                                                                                                                                                                                                             
\bibitem{Yudovich60}                                                                                                                                                                                                                         
V.~Yudovich.                                                                                                                                                                                                                                 
\newblock {Periodic motions of a viscous incompressible fluid.}                                                                                                                                                                              
\newblock {\em Sov. Math., Dokl.}, 1:168--172, 1960.                                                                                                                                                                                         
                                                                                                                                                                                                                                             
\end{thebibliography}
                                                                                                                                                                                                                         
\end{document}